\documentclass[8pt]{article}
\usepackage{indentfirst,latexsym,bm}
\usepackage{amsfonts}
\usepackage{amssymb}
\usepackage[leqno]{amsmath}
\usepackage{dsfont}
\usepackage{times}
\usepackage{amsthm}
\usepackage{hyperref}
\usepackage{color,xcolor}
\begin{document}
\title{On the minimal extension and structure of braided weakly group-theoretical fusion categories}
\author{Victor Ostrik and Zhiqiang Yu}
\date{}
\maketitle

\newtheorem{theo}{Theorem}[section]
\newtheorem{prop}[theo]{Proposition}
\newtheorem{lemm}[theo]{Lemma}
\newtheorem{coro}[theo]{Corollary}
\theoremstyle{definition}
\newtheorem{defi}[theo]{Definition}
\newtheorem{exam}[theo]{Example}
\newtheorem{remk}[theo]{Remark}
\newtheorem{conj}[theo]{Conjecture}

\newcommand{\A}{\mathcal{A}}
\newcommand{\B}{\mathcal{B}}
\newcommand{\C}{\mathcal{C}}
\newcommand{\D}{\mathcal{D}}
\newcommand{\E}{\mathcal{E}}
\newcommand{\FPdim}{\text{FPdim}}
\newcommand{\I}{\mathcal{I}}
\newcommand{\K}{\mathds{k}}
\newcommand{\M}{\mathcal{M}}
\newcommand{\N}{\mathcal{N}}
\newcommand{\Q}{\mathcal{O}}
\newcommand{\Rep}{\text{Rep}}
\newcommand{\W}{\mathcal{W}}
\newcommand{\Y}{\mathcal{Z}}
\newcommand{\Z}{\mathbb{Z}}
\theoremstyle{plain}

\abstract
We   show that any slightly degenerate weakly group-theoretical fusion category admits a minimal non-degenerate extension. Let  $d$ be a positive square-free integer,  given a  weakly group-theoretical non-degenerate  fusion category $\C$, assume that $\FPdim(\C)=nd$ and $(n,d)=1$. If $(\FPdim(X)^2,d)=1$ for all simple objects $X$ of $\C$, then we   show that $\C$ contains a non-degenerate fusion subcategory $\C(\Z_d,q)$. In particular, we obtain that integral  fusion categories of  Frobenius-Perron dimensions $p^md$ such that $\C'\subseteq \text{sVec}$ are  nilpotent and group-theoretical, where $p$ is a prime and  $(p,d)=1$.


\section{Introduction}
In this paper, we always work over an algebraically closed field $\K$  of characteristic zero.

A fusion category $\C$ is a semisimple  $\K$-linear finite abelian tensor category, we use $\Q(\C)$ to denote the set of isomorphism classes of simple objects of $\C$. Fusion category $\C$ is a braided fusion category if $\C$ is equipped with  a braiding $c$. Let $\D\subseteq\C$ be a fusion subcategory of braided fusion category $\C$. Recall that  the centralizer  $\D_\C'$ of $\D$ in $\C$ is the fusion subcategory generated by simple objects $Y$ of $\C$ such that $c_{Y,X}c_{X,Y}=\text{id}_{X\otimes Y}$, for all  objects $X\in\Q(\D)$, see \cite{Mu}; in particular, we call $\C':=\C_\C'$ the M\"{u}ger center of $\C$.

Let $\E$ be an arbitrary symmetric fusion category,  i.e., $\E'=\E$. Recall that a braided fusion category over $\E$
is a braided fusion category $\C$ equipped with a braided tensor embedding $\E \to \C'$, see e.g. \cite[2.9]{DNO}.
A braided fusion category over $\E$ is non-degenerate over $\E$ if the embedding $\E \to \C'$ is an equivalence.
In the special cases when $\E=\text{Vec}$ or $\E=\text{sVec}$ is the category of finite dimensional vector or super vector spaces over $\K$, this specializes to the notions of non-degenerate and slightly degenerate braided fusion
categories.

A non-degenerate braided fusion category $\A$ is called a minimal extension of a braided fusion category $\C$ if $\C\subseteq\A$ and $\C_\A'=\C'=:\E$. It was conjectured that these minimal extensions exist for an arbitrary braided fusion category $\C$ \cite[Conjecture 5.2]{Mu}. However, counterexamples posed by Drinfeld \cite{Drinfeld} shows that this conjecture fails for some non-degenerate fusion categories over certain Tannakian fusion categories  $\E$ (recall that $\E$ is Tannakian if it is braided equivalent to representation category $\Rep(G)$ for some
finite group $G$, see e.g. \cite{DrGNO2}). The following special case of the minimal extension conjecture
is still open and is of great interest:

\begin{conj}\footnote{This conjecture is proved in a recent paper \cite{JR}. We hope that our treatment of the special case still will be useful for the readers.} \label{minextconj}
Any slightly degenerate fusion category admits a minimal extension.
\end{conj}

Note that it is proved in \cite{LKW} and \cite{BGHNPRW} that the set of all minimal extensions (up to a suitable equivalence) of
a given slightly degenerate category is a torsor over group ${\mathbb Z}/16$. Conjecture \ref{minextconj}
states that this torsor is non-empty.


Our first main result is a proof of Conjecture \ref{minextconj} for a class of slightly degenerate fusion categories. Recall that a fusion category $\C$ is weakly group-theoretical if $\C$ is Morita equivalent to a nilpotent fusion category $\D$ \cite{ENO3} (see section \ref{section2} for the definition of nilpotency of fusion categories), equivalently, there exists  a braided equivalence between Drinfeld centers $\Y(\C)\cong\Y(\D)$ \cite[Theorem 1.3]{ENO3}.
\begin{theo}[Theorem \ref{theo2sec3}]\label{maintheorem1}Every slightly degenerate weakly group-theoretical fusion category admits a minimal  extension.
\end{theo}
In fact, let  $\C$ be a braided  fusion category such that $\C'=\E$. In Theorem \ref{theo1sec3}, we show $\C$ admits a minimal extension if and only if $\C$ is Witt equivalent to $\E\boxtimes\D$ for some non-degenerate fusion category $\D$, the precise definition of Witt equivalence can be found in section \ref{section2}.

Our second main result is  to give a sufficient condition when a  non-degenerate fusion category of Frobenius-Perron dimension $nd$ contains a pointed non-degenerate fusion subcategory of Frobenius-Perron dimension $d$, where $n$ is a positive integer  and $d$   a square-free integer such that $(n,d)=1$. Explicitly,
\begin{theo}[Theorem \ref{theo2sec4}]\label{maintheorem2} Assume that $\C$ is a  non-degenerate fusion category with $\FPdim(\C)=nd$, if $\C$ contains a Tannakian subcategory $\E=\Rep(G)$ satisfying  $\C_G^0\cong\C(\Z_d,q)\boxtimes\A$, moreover for any object $V\in\Q(\C)$, $(\FPdim(V)^2,d)=1$. Then $\C\cong\C(\Z_d,q)\boxtimes\C(\Z_d,q)_\C'$, where $\C(\Z_d,q)$ is the non-degenerate fusion category determined by metric group $(\Z_d,q)$, and $\C_G^0$ is the de-equivariantization of $\E_\C'$ by $\E$.
\end{theo}
According to \cite[Theorem 3.1]{Na}, integral non-degenerate weakly group-theoretical fusion categories $\C$ of Frobenius-Perron dimension $nd$ satisfy conditions of Theorem \ref{maintheorem2}, see Theorem \ref{theo1sec4}. Besides, combined with  Theorem \ref{maintheorem1}, we can also extend the conclusion of Theorem \ref{maintheorem2} to slightly degenerate weakly group-theoretical fusion categories, see Corollary \ref{coro1sec4}.

The paper is organized  as follows. In section \ref{section2},  we recall some basic
properties of fusion categories that we  use throughout. In section \ref{section3}, we  give a partial  answer on minimal extension of slightly degenerate fusion categories  in Theorem \ref{theo2sec3}. In section \ref{section4}, we classify  non-degenerate weakly group-theoretical fusion categories of particular Frobenius-Perron dimensions in Theorem \ref{theo2sec4}. In particular, given a prime $p$ such that $(p,d)=1$, if $\C'\subseteq \text{sVec}$, then we obtain group-theoretical property of integral fusion categories $\C$ of Frobenius-Perron dimensions $p^md$ in Corollary \ref{coro2sec4}.

This paper was written during a visit of the second named author at University of Oregon supported by China Scholarship Council (Grant no.201806140143), he  appreciates the Department of Mathematics for their warm hospitality.
The work of V.~O. was partially supported by the HSE University Basic Research Program, Russian Academic Excellence Project '5-100' and by the NSF grant DMS-1702251. Z. Yu was supported by NSFC (no.12101541), NSF of Jiangsu Province (no.BK20210785), and Natural Science Foundation of Jiangsu Higher Institutions of China (no.21KJB110006). We also thank the referees for useful comments.

\section{Preliminaries}\label{section2}
In this section, we recall   some definitions and properties of fusion categories and braided fusion categories, we refer the reader to \cite{DrGNO2,EGNO,ENO1,ENO3}.  Let $\K^{*}:=\K\backslash {\{0}\}$, $\Z_r:=\Z/r$, $r\in \mathbb{N}$.
\subsection{Fusion categories}
Let $\C$ be a fusion category,  the cardinality of $\Q(\C)$ is called rank of $\C$, and denoted by $\text{rank}(\C)$. Let $\text{Gr}(\C)$ be the Grothendieck ring of $\C$. Then  there is a unique ring  homomorphism  FPdim(-)  from the Grothendieck ring $\text{Gr}(\C)$ to $\K$ such that $\text{FPdim}(X)\geq1$ is an algebraic integer for all objects $X\in\Q(\C)$
\cite[Theorem 8.6]{ENO1}, and $\text{FPdim}(X)$ is called the Frobenius-Perron dimension of object $X$. The Frobenius-Perron dimension $\text{FPdim}(\C)$ of fusion category $\C$ is defined by
\begin{align*}
\text{FPdim}(\C):=\sum_{X\in\Q(\C)}\text{FPdim}(X)^2.
\end{align*}

A simple object $X$ of fusion category $\C$ is invertible  if $\FPdim(X)=1$. Fusion category $\C$ is a pointed fusion category, if simple objects of $\C$ are all invertible.  In this case, $\C\cong\text{Vec}_G^\omega$, the category of finite-dimensional $G$-graded vector spaces of $\K$, where $G=\Q(\C)$ is a finite group induced by fusion rules of $\C$ and $\omega\in Z^3(G,\K^*)$ is a $3$-cocycle. In the following, we use $G(\C)$ to denote the group of the isomorphism classes of invertible objects of fusion category $\C$.

A fusion category $\C$ is weakly integral if $\FPdim(\C)\in\Z$; $\C$ is integral if $\FPdim(X)\in\Z$, $\forall X\in\Q(\C)$. Let  $\C_\text{ad}$ be the  adjoint fusion subcategory of $\C$, i.e., $\C_\text{ad}$  is  generated by simple objects $Y$ such that $Y\subseteq X\otimes X^*$ for some simple object $X\in\C$. If $\C$ is weakly integral then $\C_\text{ad}$ is integral \cite[Proposition 8.27]{ENO1}.

Let $G$ be a finite group, for a $G$-graded fusion category $\C=\oplus_{g\in G}\C_g$, we use the following notation to denote the Frobenius-Perron dimension of  component $\C_g$
\begin{align*}
\FPdim(\C_g):=\sum_{X\in\Q(\C_g)}\FPdim(X)^2.
 \end{align*}
 If the $G$-grading of $\C$ is faithful, that is, for any  $g\in G$, the component $\C_g\neq 0$, then
 \begin{align*}
 \FPdim(\C)=|G|\FPdim(\C_g)
 \end{align*}
 by \cite[Proposition 8.20]{ENO1}. Any fusion category  $\C$ has a faithful grading with $\C_\text{ad}$ being its trivial component, the grading is called the universal grading of $\C$. Moreover,  for any  faithful $G$-grading of $\C=\oplus_{g\in G}\C_g$, we have $\C_\text{ad}\subseteq\C_{e}$ \cite [Corollary 3.7]{GN}.

A fusion category $\C$ is said to be nilpotent if there exists  a natural number $n$ such that $\C^{(n)}=\text{Vec}$ \cite{GN}, where $\text{Vec}$ is the category of finite-dimensional vector spaces over $\K$, and
\begin{align*}
\C^{(0)}:=\C, ~\C^{(1)}:=\C_\text{ad},~ \C^{(j)}:=(\C^{(j-1)})_\text{ad}, ~j\geq1.
\end{align*} Obviously, pointed fusion categories are nilpotent, and fusion categories of prime power Frobenius-Perron dimensions are also nilpotent \cite [Theorem 8.28]{ENO1}, for example. A fusion category $\C$ is weakly group-theoretical if it is Morita equivalent to a nilpotent
fusion category $\D$; $\C$ is group-theoretical if it is Morita equivalent to a pointed  fusion category, see \cite{ENO3} for details.
\subsection{Braided fusion categories}
A braided fusion category $\C$ is a fusion category  with a braiding
\begin{align*}
c_{X,Y}:X\otimes Y \overset{\sim}{\to}Y\otimes X, ~\forall X,Y\in\C.
 \end{align*}
Two objects $X,Y\in\C$ are said to centralize each other, if
\begin{align*}
c_{Y,X}c_{X,Y}=\text{id}_{X\otimes Y}.
\end{align*}
Let $\D\subseteq\C$ be a fusion subcategory. Then the centralizer  of $\D$ in $\C$ is the fusion subcategory generated by objects of $\C$ that centralize every object  of $\D$.
 We call $\C':=\C_\C'$ the M\"{u}ger center of $\C$ \cite{Mu}.  Moreover,  braided fusion category $\C$ is a pre-modular fusion category if $\C$ is a spherical fusion category \cite{EGNO}.

A braided fusion category $\E$ is symmetric if $\E=\E'$. For any symmetric fusion category $\E$, there exists a finite group $G$ and a central element $u\in G$   such that $\E\cong\Rep(G,u)$ \cite{De}, where $\Rep(G,u)$ is the category of finite-dimensional representation of $G$, $u^2=e$ and $u$ acts as parity automorphism for any $X\in\Q(\E)$. A symmetric fusion category $\E$ is a Tannakian fusion category if $\E\cong \Rep(G)$, the braiding of $\Rep(G)$ is given by interchange of vector spaces. A Tannakian subcategory $\E\subseteq\C$ is maximal, if $\E$ is not contained in any other Tannakian subcategory of $\C$.

We say that $\C$ is a fusion category over $\E$ if there exists a braided tensor functor $F:\E\to\Y(\C)$ such that $\E$ is mapped faithfully  to  $\C$ via functor $\text{Forg}\circ F$, where $\text{Forg}:\Y(\C)\to\C$ is the forgetful functor \cite[Definition 4.16]{DrGNO2}. A fusion category $\C$ is a braided fusion category over $\E$, if $\C$ is braided and $\E\subseteq\C'$. In particular, $\C$  is a non-degenerate fusion category over $\E$ if $\C'=\E$.

Given a braided fusion category $\C$, let $\E=\Rep(G)\subseteq\C$ be a non-trivial Tannakian subcategory, then there exists a $G$-graded fusion category $\C_G$, which is called the de-equivariantization of $\C$ by $\E$, such that  $\C\cong (\C_G)^G$ as  braided fusion categories. In general, $\C_G$ is  a braided $G$-crossed fusion category,  but the trivial component $\C_G^0\cong (\E'_\C)_G$ is a braided fusion category, where $\E_\C'$ is the centralizer of $\E$ in $\C$. In particular, if $\E\subseteq \C'$, then $\C_G$ is a braided fusion category, and $(\C_G)'\cong (\C')_G$, see \cite[section 4]{DrGNO2}.

Let $\C$ be  a braided fusion category. Then $\C$ is said to be non-degenerate if its M\"{u}ger center  $\C'=\text{Vec}$. A pre-modular fusion category $\C$ is   a modular fusion category if $\C$  is non-degenerate.
It is  known that pointed non-degenerate fusion categories are in bijective correspondence with metric groups $(G,q)$, where $G$ is an abelian group, $q: G\to \K^*$ is a non-degenerate quadratic form, see \cite{DrGNO2, EGNO} for details. We use $\C(G,q)$ to denote the pointed non-degenerate fusion category determined by the metric group $(G,q)$ below.

A braided fusion category $\C$ is slightly degenerate if $\C'=\text{sVec}$, where $\text{sVec}$ is the symmetric category of finite-dimensional super vector spaces over $\K$. We say a  slightly degenerate fusion category $\C$ splits, if there exists a braided equivalence $\C\cong \text{sVec}\boxtimes\D$, where $\D$ is  a non-degenerate fusion category. Slightly degenerate pointed fusion categories split \cite[Proposition 2.6]{ENO3}, for example. $\C$ is a super-modular fusion category if $\C'=\text{sVec}$ and $\C$ is spherical. In particular, for any slightly degenerate fusion category $\C$, if $\C$ contains a Tannakian subcategory $\E=\Rep(G)$, then $\C_G^0$ is also slightly degenerate \cite[Propositon 4.30]{DrGNO2}.

\subsection{Witt equivalence and Witt group}
Let $\C$ be a braided fusion category, $(A,m)\in\C$   a commutative algebra, where $m:A\otimes A\to A$ is the multiplication such that $m=m\circ c$.   $A$ is a connected \'{e}tale algebra if $\text{Hom}_\C(I,A)\cong\K$ and  the category of right $A$-modules $\C_A$ in $\C$ is a semisimple category \cite[Proposition 2.7, Definition 3.1]{DMNO}. In addition, we say two  \'{e}tale algebras $A,B\in\C$ are centralizing each other if $c_{B,A}\circ c_{A,B}=\text{id}_{A\otimes B}$.

Let $\E$ be a symmetric fusion category. By \cite[Proposition 3.22]{DMNO},  the  symmetric category $\E\boxtimes\E$ contains a connected \'{e}tale algebra $A$ such that $(\E\boxtimes\E)_A\cong\E$ as braided fusion category.  In fact, the \'{e}tale algebra \begin{align*}
A=\I(I):=\oplus_{X\in\Q(\E)}X\boxtimes X^*,
\end{align*}
where  $\I$ is the right adjoint functor of the surjective braided tensor functor
\begin{align*}
\otimes:\E\boxtimes\E\to\E, ~X\boxtimes Y\mapsto X\otimes Y, ~\forall X,Y\in\E.
 \end{align*}Then for any braided fusion categories $\C,\D$ over $\E$,
\begin{align*}
\C\boxtimes_\E\D:=(\C\boxtimes\D)_A,
\end{align*}
  is also a braided fusion category over $\E$, where $(\C\boxtimes\D)_A$ is the category of right $A$-modules of $A$ in $\C\boxtimes\D$.

Recall that two non-degenerate fusion categories $\C,\D$ over $\E$ are Witt equivalent if and only if there exists a fusion category $\A$  over $\E$  such that there is a braided equivalence
\begin{align*}
\C\boxtimes_\E\D^\text{rev}\cong\Y(\A,\E),
 \end{align*}
 where    $\Y(\A,\E)$ is   the centralizer  of $\E$ in Drinfeld center $\Y(\A)$, $\D^\text{rev}:=\D$ as fusion category, the braiding
 $c^\text{rev}$ of $\D^\text{rev}$ is given by
 \begin{align*}
 c^\text{rev}_{X,Y}:=c_{Y,X}^{-1}, ~\forall X,Y\in\D^\text{rev}.
  \end{align*}The Witt equivalence class of $\C$ will be denoted by $[\C]$, and the Witt equivalence classes of non-degenerate fusion categories   over $\E$ form  the Witt   group $\W(\E)$, we denote $\W:=\W(\text{Vec})$ and $s\W:=\W(\text{sVec})$, see \cite{DMNO,DNO} for detail.

Given a braided fusion category $\C$ with $\C'\cong\E$,   M. M\"{u}ger   conjectured that there is a non-degenerate fusion category $\A$ such that $\C\subseteq\A$ as braided fusion subcategory and $\C_\A'\cong\E$ \cite[Conjecture 5.2]{Mu}, and $\A$ is called a minimal extension of $\C$. Obviously, if a slightly degenerate category $\C$ splits, then it has a minimal  extension. This fails for some non-trivial Tannakian fusion category $\E$ by V. Drinfeld \cite{Drinfeld}, however.
Let $\C$ be a  super-modular fusion category, assume that $\C$ admits a minimal modular extension $\A$, i.e., the minimal extension $\A$ of $\C$ is a modular fusion category, then it was showed in \cite{BGHNPRW,LKW} that   $\A$ has exactly 16  modular equivalence classes, and these minimal modular extensions are also Witt inequivalent \cite{DNO}.

\section{On the minimal   extension conjecture}\label{section3}
In this section,   we  study the minimal extensions for braided fusion categories, and we show that every slightly degenerate weakly group-theoretical fusion category admits a minimal extension.

To begin with, we need the following lemma, which  is a direct result of
\cite[Corollary 3.26]{DMNO}.
\begin{lemm}\label{lemm1sec3}Let $\C$ be a non-degenerate fusion category over $\E$, and $\D$   a  braided fusion category. Then we have a braided equivalence $\C\boxtimes\D\cong\C\boxtimes_\E (\E\boxtimes\D)$.
\end{lemm}
\begin{proof}Notice that we have an injective braided tensor functor
\begin{align*}
\iota:\C\boxtimes\D\to \C\boxtimes (\E\boxtimes\D),~X\boxtimes Y\mapsto X\boxtimes I\boxtimes Y,
 \end{align*}
 and a surjective braided tensor functor $F:\C\boxtimes (\E\boxtimes\D)\to\C\boxtimes_\E(\E\boxtimes\D)$. Then the composition $F\circ\iota:\C\boxtimes\D\to\C\boxtimes_\E (\E\boxtimes\D)$ is a braided tensor functor. By definition, we   have
 \begin{align*}
 \C\boxtimes_\E(\E\boxtimes\D)=(\C\boxtimes\E\boxtimes\D)_A
  \end{align*}
  where $A=\I(I)=\oplus_{X\in\Q(\E)}X\boxtimes X^*\boxtimes I$, and $\I$ is the right adjoint functor of tensor functor $\otimes:\E\boxtimes\E\to\E$. Since $(\C\boxtimes\D)\cap A\cong I$,  $F\circ\iota$ is injective by \cite[Proposition 3.4]{LKW}. Hence, $F\circ\iota$  is a braided equivalence by \cite[Proposition 6.3.3]{EGNO}, as we have the following equations  \begin{align*}
\FPdim(\C\boxtimes\D)=\FPdim(\C\boxtimes_\E (\E\boxtimes\D))=\FPdim(\C)\FPdim(\D),
\end{align*}
this finishes the proof of the lemma.
\end{proof}

For any connected \'{e}tale algebra $A\in\C$, a right $A$-module $(M,\mu)$ is a dyslectic (or, local) $A$-module, if $\mu\circ c_{A,M}\circ c_{M,A}=\mu$, where $\mu:M\otimes A\to M$ is the $A$-module structure of $M$. We use $\C_A^0$ to denote the category of dyslectic  modules of $A$ \cite{DMNO}.
\begin{theo}\label{theo1sec3}Let $\B$ and $\C$ be two Witt equivalent non-degenerate  fusion categories over a symmetric fusion category $\E$,  moreover $\B\cong\E\boxtimes\D$ where $\D$ is non-degenerate. Then $\C$ admits a minimal extension $\A$. In addition, $[\C]=[\E\boxtimes\A]$.
\end{theo}
\begin{proof} By assumption,  there exists a fusion category $\A_1$  over $\E$ such that there is a braided equivalence $\C\boxtimes_\E\B^\text{rev}\cong\Y(\A_1,\E)$. On the one hand, by Lemma \ref{lemm1sec3}, we have the following braided equivalences
\begin{align*}
\C\boxtimes\D^\text{rev}\cong\C\boxtimes_\E\B^\text{rev}\cong \Y(\A_1, \E).
 \end{align*}
 On the other hand,  \cite[Theorem 3.13]{DrGNO2} says that $\Y(\A_1)\cong\D^\text{rev}\boxtimes\A$, where $\A$ is the  centralizer of $\D^\text{rev}$ in $\Y(\A_1)$, and then $\A$ is also non-degenerate. By \cite[Theorem 3.10]{DrGNO2}, \begin{align*}
 \FPdim(\Y(\A_1))=\FPdim(\E)\FPdim(\Y(\A_1, \E)),
 \end{align*}
  hence we have equations
\begin{align*}
\FPdim(\A)=\frac{\FPdim(\Y(\A_1))}{\FPdim(\D)}=\frac{\FPdim(\E)\FPdim(\Y(\A_1, \E))}{\FPdim(\D)}=\FPdim(\E)\FPdim(\C).
\end{align*}

We claim that $\A$ is a minimal   extension of $\C$. Note that we have an injective braided tensor functor from $\C\boxtimes \D^\text{rev}$ to $\Y(\A_1)$,  by definition $\C$ and $\D^\text{rev}$ are centralizing each other in braided fusion category $\C\boxtimes \D^\text{rev}$, hence they are still centralizing each other in  $\Y(\A_1)$. Therefore, we have an inclusion $\D^\text{rev}\subseteq\C'_{\Y(\A_1)}$, where $\C'_{\Y(\A_1)}$ is the centralizer of $\C$ in $\Y(\A_1)$. Since the centralizer of $\D^\text{rev}$ in $\Y(\A_1)$ is $\A$, by \cite[Corollary 3.11]{DrGNO2} $\C \subseteq\A$, as claimed.

Let $A$ be the connected \'{e}tale algebra such that $(\E\boxtimes\E)_A\cong\E$. Denote $\M:=\E\boxtimes\A$. Since $A\cap\E=I$,  $\M_A^0$ is also non-degenerate over $\E$ by \cite[Corollary 4.6]{DNO}, and we have the following braided equivalence
\begin{align*}
\M\boxtimes_\E(\M_A^0)^\text{rev}\cong\Y(\M_A,\E).
\end{align*}
In particular, $\FPdim(\M_A^0)=\frac{\FPdim(\M)}{\FPdim(\E)^2}=\FPdim(\C)$. Notice that $\C\cong\text{Vec}\boxtimes\C$ as a fusion subcategory of $\M$ and $\C\cap A=I$, then there is an injective braided tensor functor $F:\C\to\M_A^0$ \cite[Proposition 3.4]{LKW},  which must be an equivalence by \cite[Proposition 6.3.3]{EGNO}. That is,
\begin{align*}
\M\boxtimes_\E\C^\text{rev}\cong\Y(\M_A,\E),
\end{align*}
 so $[\E\boxtimes\A]=[\C]$ by definition, this finishes the proof.
\end{proof}

\begin{remk}
Let $\C$ and $\D$ be Witt equivalent  braided  fusion categories. If $\C$ has a minimal  extension $\A$, then Theorem \ref{theo1sec3} implies that $\D$ is also Witt equivalent to $\E\boxtimes\A$.   Hence, $\D$ has a minimal extension. Consequently, the property whether a braided fusion category admits a minimal extension is a Witt invariant.

In addition, if $\C$ is a pre-modular fusion category, we don't know whether its minimal extension $\A$ admits a spherical structure, however.
\end{remk}

Therefore, to prove minimal  extension conjecture for non-degenerate fusion categories over a symmetric fusion category $\E$ is equivalent   to show that the group homomorphism $S_\E:\W\to \W(\E)$ is surjective, where
\begin{align*}
S_\E:[\A]\mapsto[\E\boxtimes\A],~ \text{for all}~[\A]\in\W,
 \end{align*}
 $[\A]$ is the Witt equivalence class of non-degenerate fusion category $\A$, see also \cite[Question 5.15]{DNO}. When $\E=\text{sVec}$, it is well-known $\text{ker}(S_\text{sVec})=\W_\text{Ising}\cong\Z_{16}$ \cite[Proposition 5.14]{DNO}, where $\W_\text{Ising}\subseteq\W$ is the subgroup generated by the Witt equivalence class of Ising fusion category.

Assume that $\E=\Rep(G)$ is a Tannakian fusion category, then for any non-degenerate fusion category $\C$ over $\E$, $\C_G$ is a non-degenerate fusion category \cite[Proposition 4.56]{DrGNO2}. Thus,   there is a well-defined group homomorphism $\phi_\E$ between Witt groups $\W(\E)$ and $\W$, where
\begin{align*}
\phi_\E:\W(\E)\to\W, ~[\C]\mapsto[\C_G],~ \text{for all}~ [\C]\in\W(\E),
\end{align*}
see \cite[Remark 5.8]{DNO}. For any non-degenerate fusion category $\A$, note that $(\E\boxtimes\A)_G\cong\A$ as braided fusion category, so $\phi_\E(S_\E([\A]))=[\A]$. Therefore,
\begin{coro}Let $\E=\Rep(G)$ be a Tannakian fusion category. Then $\phi_\E\circ S_\E=\text{id}$, so there exists a split exact sequence
\begin{align*}
1\to \text{ker}(\phi_\E)\overset{i}{\to} \W(\E)\underset{S_\E}{\overset{\phi_\E}{\rightleftarrows} }\W\to 1.
\end{align*}
\end{coro}
Hence $S_\E$ is injective for all Tannakian fusion categories $\E$. However, Drinfeld's counterexample \cite{Drinfeld} shows  that $S_\E$  is not surjective for some Tannakian categories $\E$, in general. In \cite{OsYu}, we will continue to consider the  structures  of   subgroup $\text{ker}(\phi_\E)$ and Witt group $s\W$.

Recall that a braided fusion category $\C$ is said to be anisotropic if  $\C$ does not contain any non-trivial Tannakian subcategory; meanwhile $\C$ is weakly anisotropic, if $\C$ does not contain a non-trivial Tannakian subcategory (if exists) that  is  invariant under all braided auto-equivalences of $\C$ \cite{DrGNO2}.
\begin{theo}\label{theo2sec3}Let $\C$ be a slightly degenerate weakly group-theoretical fusion category. Then $\C$ admits a minimal extension.
\end{theo}
\begin{proof} By Theorem \ref{theo1sec3}, it suffices to show that $\C$ is Witt equivalent to a split   slightly degenerate fusion category. If $\C$ is   pointed, then \cite[Proposition 2.6]{ENO3} means  $\C\cong\text{sVec}\boxtimes\D$, where $\D$ is a pointed non-degenerate fusion category. Assume  $\C$ is not a  pointed fusion category below.

On the one hand, if $\C$ is weakly anisotropic, then $\C$ is braided equivalent to $\B\boxtimes\D$ by \cite[Theorem 1.1]{Na}, where $\B\cong \text{Vec}$  or $\B$ is equivalent to an Ising category $\I$, $\D$ is a pointed fusion  category. So $\B$ is non-degenerate and $\D$ is slightly degenerate. While \cite[Proposition 2.6]{ENO3} implies  $\D\cong \text{sVec}\boxtimes\D_0$, where $\D_0$ is a  pointed non-degenerate fusion category, therefore $\C \cong \text{sVec}\boxtimes(\B\boxtimes\D_0)$. That is,  $\C$ splits and admits a minimal extension.

On the other hand, if $\C$ is not  a weakly anisotropic fusion category, then  $\C$ contains a maximal Tannakian subcategory $\Rep(G)$ such that $\C_G^0$ is a  slightly degenerate weakly anisotropic fusion category by
\cite[Corollary 5.19]{DrGNO2}. Meanwhile, we have a braided fusion categories tensor equivalence $\C\boxtimes_{\text{sVec}}(\C_G^0)^\text{rev}\cong\Y(\C_G, \text{sVec})$ by \cite[Corollary 4.6] {DNO}. Thus, $\C$ and $\C_G^0$ are Witt equivalent by definition, consequently $\C$ admits a minimal extension by Theorem \ref{theo1sec3}.
\end{proof}

\begin{remk}Let $\D$ be a slightly degenerate fusion category. If  $\D$ is Witt equivalent to a weakly group-theoretical fusion category, then Theorem \ref{theo1sec3} and Theorem \ref{theo2sec3} together show   $\D$ admits a minimal extension.

In addition,  the proof of  Theorem \ref{theo1sec3} also implies that the minimal extension $\A$ in  Theorem \ref{theo2sec3} is a braided fusion subcategory of $\Y(\C_G)$. Hence $\A$   is   a weakly group-theoretical fusion category \cite[Proposition 4.1]{ENO3}. In particular, $\A$ is integral if $\C$ is integral.
\end{remk}
\section{Structure  of certain braided weakly group-theoretical fusion categories}\label{section4}
In this section, we study the structure of non-degenerate and slightly degenerate weakly group-theoretical fusion categories of particular Frobenius-Perron dimensions.

Let $G$ be a finite group. Given a faithful $G$-graded fusion category $\C=\oplus_{g\in G}\C_g$, assume    $\C_e\cong\text{Vec}_N^\omega$, where $N$ is a finite group and $\omega\in Z^3(N,\K^*)$ is a $3$-cocycle. By definition, the  $G$-grading structure of $\C$ induces  an action of group $N$ on sets $\Q(\C_g)$ for all $g\in G$.

The following lemma is trivial, we include it for the reader's interest.
\begin{lemm}\label{lemm1sec4}Let $\C=\oplus_{g\in G}\C_g$ be a faithful $G$-graded  fusion category, and $\C_e\cong\text{Vec}_N^\omega$ is pointed, where $G,N$ are finite groups and $\omega\in Z^3(N,\K^*)$ is a $3$-cocycle. Then for any element $ g\in G$,  N acts transitively on   $\Q(\C_g)$.
\end{lemm}
\begin{proof} By definition, for any $g\in G$, let $X\in\C_g$ be an arbitrary simple object, there exists  a subgroup $H$ of $N$ such that $X\otimes X^*=\oplus_{h\in H}h$, then $\FPdim(X)^2=|H|$. Note that the orbit of $X$  has a size of $[N:H]$,  by \cite[Proposition 8.20]{ENO1}
\begin{align*}
\FPdim(\C_g)=\FPdim(\C_e)=|N|=[N:H]\FPdim(X)^2,
  \end{align*}
  which then shows   the orbit of $X$ is exactly  $\Q(\C_g)$.
\end{proof}
Now we are ready to give our first classification theorem of non-degenerate fusion categories.
\begin{theo}\label{theo1sec4}Let  d be a square-free positive integer. Assume that $\C$ is  an integral non-degenerate weakly group-theoretical fusion category  such that  $\FPdim(\C)=nd$ and $(n,d)=1$. Then $\C\cong \C(\Z_d,q)\boxtimes \C(\Z_d,q)_\C'$ as braided fusion category.
\end{theo}
\begin{proof}
Assume that $\C$ is not pointed, otherwise the result is trivial. Since $\C$ is a weakly group-theoretical integral fusion category, $\C$ contains a non-trivial Tannakian subcategory by \cite[Theorem 3.1]{Na}. Let $\E=\Rep(G)$ be a maximal Tannakian subcategory of $\C$, then  $\C_G$ is a $G$-crossed braided fusion category  whose trivial component is braided equivalent to  $\C_G^0$; moreover, the $G$-grading of $\C_G$ is faithful and $\C_G^0$ is a non-degenerate fusion category by \cite[Proposition 4.56]{DrGNO2}.  Meanwhile, \cite[Corollary 5.19]{DrGNO2} says $\C_G^0$ is  weakly anisotropic. Therefore,    $\C_G^0$ is a pointed fusion category
by \cite[Theorem 1.1]{Na}.  So $\C_G^0$ contains a pointed non-degenerate  fusion category $\C(\Z_d,q)$, since $|G|^2$ divides $\FPdim(\C)$.

Let $\B:=\C_G$ and $\B=\oplus_{g\in G}\B_g$ be the $G$-grading of $\B$. Then $\B_\text{ad}\subseteq\B_e$ by universal property of $\B_\text{ad}$ \cite [Corollary 3.7]{GN}, so $\B$ is nilpotent. Hence, for any $g\in G$ and any object $X\in\Q(\B_g)$, we have that $\FPdim(X)^2$ divides $\FPdim(\B_e)$ \cite[Corollary 5.3]{GN}, so $(\FPdim(X),d)=1$. By Lemma \ref{lemm1sec4} $\text{rank}(\B_g)=\frac{\FPdim(\B_e)}{\FPdim(X)^2}$, in particular $d| \text{rank}(\B_g)$. Since $\B_e=\C_G^0$ is a non-degenerate fusion category, by \cite[Lemma 10.7]{Kir} $\text{rank}(\B_g)$ is equal to the cardinality of $\Q(\C^0_G)^g$, the set of isomorphism classes of simple objects of $\C^0_G$ that are fixed by $g$. Note that $g$ acts as a braided tensor equivalence on $\C_G^0$, thus $\Q(\C_G^0)^g$ is an abelian group and contains $\Z_d$ as a subgroup, which then implies that $G$ acts trivially on $\C(\Z_d,q)$. For $(d,|G|)=1$ and $H^i(\Z_d,G)=0$ for all non-negative integer $i$, fusion category $\E'=(\C_G^0)^G\subseteq\C$ contains a fusion subcategory $\C(\Z_d,q)^G\cong \Rep(G)\boxtimes \C(\Z_d,q)$.

Therefore,  $\C$ contains a non-degenerate fusion subcategory $\C(\Z_d,q)$.
By \cite[Theorem 3.13]{DrGNO2} we have a braided tensor equivalence $\C\cong\C(\Z_d,q)\boxtimes\D$, where $\D$ is the centralizer of $\C(\Z_d,q)$, and $\D$ is also a  non-degenerate fusion category.
\end{proof}

Next we will generalize Theorem \ref{theo1sec4} to strictly weakly integral non-degenerate fusion categories. Assume that a braided weakly group-theoretical fusion category $\C$  is  strictly weakly integral. However,  if $\C$ is weakly anisotropic,  it was shown in  \cite[Theorem 1.1]{Na} that $\C$ needs not be a pointed fusion category. Hence, the proof of Theorem \ref{theo1sec4} fails in general.

Let $\E=\Rep(G)$ be a Tannakian fusion category. Assume that
\begin{align*}
B:=\I(I)=\oplus_{X\in\Q(\E)}X\boxtimes X^*
 \end{align*}
 is the connected \'{e}tale algebra such that $(\E\boxtimes\E)_B\cong\E$ as braided fusion categories \cite[Remark 2.11]{DNO}.
Let $\B$ and $\C$ be braided fusion  categories, which contain $\E$ as a fusion subcategory. Since $\E\boxtimes\E\subseteq \C \boxtimes\B$, we have a braided  fusion category $(\C\boxtimes\B)_B^0$ by \cite[Proposition 4.30]{DrGNO2}, we denote   it by $\C\widehat{\boxtimes}_\E\B$, for simplification  of notation.

Given two $G$-crossed braided fusion categories $\C_G$ and $\D_G$, notice that the fusion category $\C_G\boxtimes\B_G$ is graded by group $G\times G$ with trivial component $\C_G^0\boxtimes\B_G^0$. Let  $\C_G\boxtimes_G\B_G$ denote the fusion subcategory of $\C_G\boxtimes\B_G$ graded by group $H$ of diagonal elements of $G\times G$, so $H\cong G$. Then there  exists a well-defined action of $G$ on the braided fusion subcategory $\C_G^0\boxtimes\B_G^0$, and naturally $\C\widehat{\boxtimes}_\E\B\cong (\C_G\boxtimes_G\B_G)^G$ as braided fusion categories by \cite[Theorem 7.12]{ENO2}.

In addition, $\C\boxtimes_\E\B$ contains a Tannakian subcategory $(\E\boxtimes\E)_B\cong \Rep(G)$, then it follows from \cite[Lemma 4.32]{DrGNO2} that there is a  braided fusion category equivalence \begin{align*}
(\C\boxtimes_\E\B)_G^0\cong(\C\boxtimes\B)_{G\times G}^0,
\end{align*} which then is equivalent to $\C^0_G\boxtimes\B^0_G$.

\begin{prop}\label{prop1sec4}Let $A,B$ be two connected \'{e}tale algebras  in a non-degenerate fusion category $\C$. If $A\cap B=I$, moreover $A,B$ are centralizing each other, then we have the following braided fusion categories equivalences
\begin{align*}
(\C_B^0)_{A\otimes B}^0\cong \C_{A\otimes B}^0\cong(\C_A^0)_{A\otimes B}^0.
\end{align*}
\end{prop}
\begin{proof}As \'{e}tale algebras $A,B$ are centralizing each other, so $A\otimes B$ is also a commutative algebra; since \'{e}tale algebras are self-dual \cite[Remark 3.4]{DMNO}, $\text{Hom}_\C(I,A\otimes B)\cong\text{Hom}_\C(A,B)$ is one-dimensional, hence $A\otimes B$ is connected. It is easy to see that $A\otimes B$ is a commutative algebra over $A$, hence we have an abelian category equivalence  $(\C_A)_{A\otimes B}\cong\C_{A\otimes B}$, so $A\otimes B$ is a connected \'{e}tale algebra by \cite[Proposition 2.7, Proposition 3.16]{DMNO}.
By \cite[Corollary 3.30, Corollary 3.32]{DMNO} $(\C_A^0)_{A\otimes B}^0$ and $(\C_B^0)_{A\otimes B}^0$ are non-degenerate fusion categories with same Frobenius-Perron dimension.  It follows from \cite[Proposition 6.3.3]{EGNO} and \cite[Corollary 3.26]{DMNO} that   $\C_{A\otimes B}^0\cong(\C_A^0)_{A\otimes B}^0$ as braided fusion categories, then $(\C_B^0)_{A\otimes B}^0\cong \C_{A\otimes B}^0\cong (\C_A^0)_{A\otimes B}^0$.
\end{proof}
\begin{lemm}[\cite{LKW}]\label{lemm2sec4}Let $\C,\D$ be non-degenerate braided fusion categories. Assume that there exists  a Tannakian category $\E=\Rep(G)$ which embeds in both $\C$ and $\D$, and such that there is a $G-$equivariant braided equivalence $\C_G^0\cong\D_G^0$.
Then there is an element $\omega\in H^3(G,\K^*)$ and a braided equivalence $\C\cong\D\widehat{\boxtimes}_\E\Y(\text{Vec}_G^\omega)$ (or, equivalently, an equivalence of $G-$crossed categories
$\C_G\cong \D_G\boxtimes_G\text{Vec}_G^{\omega}$).
\end{lemm}
\begin{proof} Since the equivalence $\C_G^0\cong\D_G^0$ is $G-$equivariant, we have a braided equivalence
$(\C_G^0)^G\cong (\D_G^0)^G$. Thus the categories $\C$ and $\D$ are minimal
extensions of the same category $(\C_G^0)^G$ with M\"uger center $\E$. According to \cite[Theorem 5.4]{LKW}
any two such extensions should differ by a twist coming from the group $\M_\text{ext}(\E)$ of minimal extensions
of $\E$. The result follows from explicit description of this group in \cite[4.3]{LKW}.

Alternatively, the de-equivariantizations $\C_G$ and $\D_G$ are $G-$crossed categories which share
the same zero component $\C_G^0\cong\D_G^0$ together with $G-$action. Now the result
follows from the description of all $G-$crossed extensions in \cite[Theorem 7.12]{ENO2}.
\end{proof}
For any braided fusion category $\C$, the equivalence classes of invertible $\C$-module categories form a group, which is called the  Picard group $\text{Pic}(\C)$ of  $\C$ \cite[section 4]{ENO2}. In addition, if  $\C$ is  a non-degenerate fusion category, then \cite[Theorem 5.2]{ENO2} says that the  group $\text{Aut}^\text{br}_\otimes(\C)$ of braided tensor auto-equivalences of $\C$  is isomorphic to $\text{Pic}(\C)$.
\begin{theo}\label{theo2sec4}
Let $\C$ be a  non-degenerate fusion category with $\FPdim(\C)=nd$, where $n$ is a positive integer, and $d$ is a square-free integer such that $(n,d)=1$. Assume that $\C$ contains a Tannakian subcategory $\E=\Rep(G)$ satisfying  $\C_G^0\cong\C(\Z_d,q)\boxtimes\A$, moreover for any $V\in\Q(\C)$, $(\FPdim(V)^2,d)=1$. Then $\C\cong\C(\Z_d,q)\boxtimes\C(\Z_d,q)_\C'$ as braided fusion category.
\end{theo}
\begin{proof} The category $\C_G^0\cong\C(\Z_d,q)\boxtimes\A$ is non-degenerate by \cite[Proposition 4.56 (ii)]{DrGNO2}. It follows that both $\C(\Z_d,q)$ and $\A$ are non-degenerate. Also we have $(|G|,d)=1$
since $\FPdim(\C)=|G|^2\FPdim(\C_G^0)$.

Consider the natural action of $G$ on $\C_G^0\cong\C(\Z_d,q)\boxtimes\A$. The order of the group
of invertible objects of $\C_G^0$ divides $\FPdim(\C_G^0)$, hence it divides $\FPdim(\C)=nd$, hence
it is of the form $n_1d$ where $(n_1,d)=1$. It follows that all invertible
objects of $\C_G^0$ of order dividing $d$ are contained in $\C(\Z_d,q)$. Thus this subcategory
is preserved by the action of $G$; hence $G$ also preserves $\A=\C(\Z_d,q)'_{\C_G^0}$. Thus
the action of $G$ on $\C_G^0\cong\C(\Z_d,q)\boxtimes\A$ is isomorphic to the product of actions
on $\C(\Z_d,q)$ and on $\A$.

We are going to show the following:

(1) The actions of $G$ on $\C(\Z_d,q)$ and $\A$ give rise to $G-$crossed extensions $\widetilde{\C(\Z_d,q)}$
and $\widetilde{\A}$ of these categories,
i.e. the corresponding obstructions (see \cite[Theorem 7.12]{ENO2}) ] in $H^4(G,\K^*)$ vanish.

(2) The action of $G$ on $\C(\Z_d,q)$ is trivial, so
the action of $G$ on $\C_G^0\cong\C(\Z_d,q)\boxtimes\A$ is isomorphic to the product of trivial action
on $\C(\Z_d,q)$ and the action above on $\A$.

Assume that (1) and (2) are proved. Let $\widetilde{\A}^G$ be the equivariantization of $\widetilde{\A}$; this
is a non-degenerate braided fusion category, see \cite[Proposition 4.56]{DrGNO2}. The category
$\D=\C(\Z_d,q)\boxtimes \widetilde{\A}^G$ is also a non-degenerate braided fusion category containing
Tannakian subcategory $\E=\Rep(G)$ in ${\bf 1}\boxtimes \widetilde{\A}^G\simeq \widetilde{\A}^G$;
by (2) above we have a $G-$equivariant braided equivalence $\D_G^0\cong \C_G^0$. Thus by Lemma
\ref{lemm2sec4} we will have
$$\C\cong\D\widehat{\boxtimes}_\E\Y(\text{Vec}_G^\omega)\cong (\C(\Z_d,q)\boxtimes \widetilde{\A}^G)\widehat{\boxtimes}_\E\Y(\text{Vec}_G^\omega)\cong \C(\Z_d,q)\boxtimes (\widetilde{\A}^G\widehat{\boxtimes}_\E\Y(\text{Vec}_G^\omega)),$$
which immediately implies Theorem \ref{theo2sec4}. Thus it remains to prove (1) and (2).

Let us prove (1). For the category $\C(\Z_d,q)$ the obstruction is annihilated by multiplication $d^4$ by \cite[Theorem 8.16]{ENO2}. On the other hand, any element of $H^4(G,\K^*)$ is annihilated by
multiplication by $|G|$. Since $(|G|,d)=1$, we get that the obstruction is zero.

Let us construct $\widetilde{\A}$. Pick a $G-$crossed extension $\widetilde{\C(\Z_d,q)}$ and
let $\B$ be the category $\widetilde{\C(\Z_d,q)}^G$ with reversed braiding. Then $\B$ is non-degenerate,
it contains $\E$, and $\B_G^0=\C(\Z_d,q)^\text{rev}\cong \C(\Z_d,q^{-1})$ where $q^{-1}$ is the inverse quadratic form of $q$ on $\Z_d$, that is, $q^{-1}(v):=q(v)^{-1}$, for all $v\in\Z_d$. Now consider the product
of $G-$crossed categories $\C_G\boxtimes_G\B_G$ (the category $\C_G\boxtimes \B_G$ is naturally
graded by $G\times G$ and $\C_G\boxtimes_G\B_G$ is the part supported on the diagonal subgroup
$\Delta(G)\subset G\times G$). Clearly
$$(\C_G\boxtimes_G\B_G)^0=\C_G^0\boxtimes\B_G^0=\C(\Z_d,q)\boxtimes\A \boxtimes\C(\Z_d,q^{-1})=
\C(\Z_d\oplus \Z_d,q\oplus q^{-1})\boxtimes\A.$$
Note that the simple objects of $\C(\Z_d\oplus \Z_d,q\oplus q^{-1})$ labelled by the diagonal subgroup
$\Delta (\Z_d)\subset \Z_d\oplus \Z_d$ generate a Tannakian subcategory ${\mathcal H} \subset \C(\Z_d\oplus \Z_d,q\oplus q^{-1})$, which is invariant under $G-$action. Then $G-$equivariantization ${\mathcal H}^G$ of ${\mathcal H}$ is a Tannakian category
equivalent to $\Rep(\widetilde{G})$ for some group $\widetilde{G}$ which fits into exact sequence
$1\to\Z_d\to \widetilde{G}\to G\to 1$. Since $(|G|,d)=1$ this sequence splits, and $\widetilde{G}\cong\Z_d\rtimes G$. Thus ${\mathcal H}^G=\Rep(\widetilde{G})$ contains a connected \'{e}tale algebra $A=\text{Fun}(\widetilde{G}/G)$ ($\K-$valued functions on finite set $\widetilde{G}/G$). In other words, the direct sum of all simple
objects of $\C(\Z_d\oplus \Z_d,q\oplus q^{-1})$ labelled by the diagonal subgroup
$\Delta (\Z_d)$ admits a structure of $G-$equivariant connected \'{e}tale algebra.

Now we would like to define $\widetilde{\A}=(\C_G\boxtimes_G\B_G)_A^0$. However, this requires a generalization
of the construction of dyslectic modules to the setting of $G-$crossed categories which we could not find in existing
literature. Thus we consider the algebra $A$ as an object of non-degenerate braided fusion category $(\C_G\boxtimes_G\B_G)^G\supset ((\C_G\boxtimes_G\B_G)^0)^G$. Finally we define
$$\widetilde{\A}:=(((\C_G\boxtimes_G\B_G)^G)_A^0)_G$$
(note that $((\C_G\boxtimes_G\B_G)^G)_A^0$ contains $({\mathcal H}^G)_A^0=({\mathcal H}^G)_A=\E =\Rep(G)$, so de-equivariantization by $G$ makes sense). Let $B=\text{Fun}(G)\in \E \subset {\mathcal H}^G$ be the regular
algebra of $G$, see \cite[Example 2.8]{DMNO}. Then the algebras $A$ and $B$ satisfy the assumptions
of Proposition \ref{prop1sec4} and we get a braided equivalence
$$\widetilde{\A}^0=(((\C_G\boxtimes_G\B_G)^G)_A^0)_G^0=(((\C_G\boxtimes_G\B_G)^G)_G^0)_A^0=((\C_G\boxtimes_G\B_G)^0)_A^0=\A$$
(we use that operations $(?)_G^0$ and $(?)_B^0$ are identified, see \cite[Example 3.14]{DMNO}).
Moreover it is clear that the canonical $G-$action on $\widetilde{\A}^0=\A$ coincides with the action defined
in the beginning of the proof (since we can describe the identification $\A\cong \widetilde{\A}^0$ as follows:
start with an obvious embedding $\A \subset (\C_G\boxtimes_G\B_G)^0\subset \C_G\boxtimes_G\B_G$ and
promote it to $\A^G \subset (\C_G\boxtimes_G\B_G)^G\to ((\C_G\boxtimes_G\B_G)^0)_A$; since $\A$
and $A$ centralize each other, the image is contained in $((\C_G\boxtimes_G\B_G)^0)_A^0$, so we get
a $G-$equivariant braided tensor functor $\A=(\A^G)_G\to (((\C_G\boxtimes_G\B_G)^0)_A^0)_G$).
Thus the action of $G$ on $\A$ can be upgraded to $G-$crossed extension $\widetilde{\A}$ and (1) is proved.

Let us prove (2), i.e. that $G$ acts on $\C(\Z_d,q)$ trivially. If not, then by  \cite[Lemma 10.7]{Kir}, there exists a  non-trivial component of the $G$-crossed grading of $\widetilde{\C(\Z_d,q)}$ which has  rank less than $d$. Since $\C(\Z_d,q)$ is the trivial component of $\widetilde{\C(\Z_d,q)}$, and the $G$-grading of $\widetilde{\C(\Z_d,q)}$  is faithful, there exists non-invertible simple objects of $\widetilde{\C(\Z_d,q)}$ with Frobenius-Perron dimension  $\sqrt{t}$ by \cite[Corollary 5.3]{GN}, where $t>1$ is an integer and $t|d$. Then the $G-$crossed category
$\widetilde{\D}:=\widetilde{\C(\Z_d,q)}\boxtimes_G\widetilde{\A}$ contains a simple object $W$ with
$\FPdim(W)^2=st$ for some integer $s$. By Lemma \ref{lemm2sec4} the same is true for the category $\C_G$;
hence the category $\C=(\C_G)^G$ contains a simple object $V$ such that $\FPdim(V)^2$ is divisible by $st$.
This contradicts the assumption $(\FPdim(V)^2,d)=1$ for all simple objects in $\C$ and (2) is proved.
\end{proof}

\begin{remk} In Theorem \ref{theo2sec4}, the condition that  $(\FPdim(V)^2,d)=1$ for all simple objects $V\in \C$ can not be dropped. For example, let  $d>1$ be a square-free odd integer, there exists a non-degenerate fusion category $\C$ of Frobenius-Perron dimension $4d$, which is braided  equivalent to a $\Z_2$-equivariantization of a Tambara-Yamagami fusion category $\mathcal{TY}(\Z_d,\tau,\mu)$. It was proved that $\C$ contain a simple object  of Frobenius-Perron dimension $\sqrt{d}$ and $\C_\text{pt}=\Rep(\Z_2)$, see \cite[Theorem 3.1]{BGNPRW} for details. Obviously, $\C\ncong\C(\Z_d,q)\boxtimes\C(\Z_d,q)_\C'$ as braided  fusion category.
\end{remk}

\begin{coro}Let $\C$ be a weakly group-theoretical non-degenerate fusion category, and $d$ be a  square-free integer. Assume that  $\FPdim(\C)=nd$ and  $(n,d)=1$. If $(\FPdim(X)^2,d)=1$ for all $X\in\Q(\C)$, then $\C\cong\C(\Z_d,q)\boxtimes\C(\Z_d,q)_\C'$ as braided fusion category.
\end{coro}
\begin{proof}
 Assume $\C$ is not pointed, otherwise the result is trivial.  If $\C$ is a weakly anisotropic fusion category, it follows from \cite[Theorem 1.1]{Na} that  $\C$ is braided equivalent to a Deligne tensor product of an Ising category and a pointed fusion category. If not, let $\E=\Rep(G)$ be a maximal Tannakian subcategory of $\C$. By \cite[Corollary 5.19]{DrGNO2}  $\C_G^0$ is a weakly anisotropic fusion category, as $|G|^2$ divides $\FPdim(\C)$, thus \cite[Theorem 1.1]{Na} says that $\C_G^0$ satisfies the conditions of    Theorem \ref{theo2sec4}, hence $\C$ contains a non-degenerate fusion subcategory $\C(\Z_d,q)$.
\end{proof}
Moreover, by using  Theorem \ref{theo2sec3} and Theorem \ref{theo2sec4},  we  have the following direct result.
\begin{coro}\label{coro1sec4}Let $n$ be an integer, $d$  an odd square-free integer such that $(n,d)=1$.
Let $\C$ be a slightly degenerate weakly group-theoretical fusion category with $\FPdim(\C)=nd$. If  there exists a minimal extension   $\A$ of   $\C$ such that $(\FPdim(V)^2,d)=1$ for all $V\in\Q(\A)$, then  $\C\cong \C(\Z_d,q)\boxtimes\C(\Z_d,q)_\C'$ as   braided fusion category.
\end{coro}

Recall that a braided  fusion category $\C$ of Frobenius-Perron dimension $p^ar^bd$ is weakly group-theoretical \cite[Proposition 3.14]{Yu}, where $p,r$ are primes, $a,b$ are non-negative integers, $d$ is a square-free integer such that $(pr,d)=1$; for non-degenerate fusion categories, this conclusion was first proved in \cite[Corollary 5.4]{Na}. Therefore, if $\C'\subseteq \text{sVec}$ and $\C$ is integral, then for all simple objects $X$ of $\C$, $\FPdim(X)^2$ divides $\FPdim(\C)$   by \cite[Theorem 2.11]{ENO3} and \cite[Corollary 3.4]{Yu}, hence $(\FPdim(X),d)=1$. Thus, Theorem \ref{theo1sec4}  shows  that $\C\cong\C(\Z_m,q)\boxtimes\D$, where  $\D$ is a braided fusion  category with M\"{u}ger center $\D'\subseteq \text{sVec}$, and $m=d$ or $\frac{d}{2}$.

In particular, we have the following corollary, which  generalizes  the conclusions of \cite[Theorem 4.7]{DongNa} and \cite[Proposition 3.8, Corollary 3.12]{Yu}.
\begin{coro}\label{coro2sec4}Let $\C$ be an integral braided fusion category of Frobenius-Perron dimension $p^nd$  and $\C'\subseteq \text{sVec}$, where $p$ is a prime, $n$  is a nonnegative integer and $d$ is a  square-free integer such that $(p,d)=1$. Then $\C$ is nilpotent and group-theoretical.
\end{coro}

\begin{proof} If $\C'=\text{Vec}$, then $\C\cong \D\boxtimes \C(\Z_d,q)$ as braided fusion category by Theorem \ref{theo2sec4}, and $\FPdim(\D)=p^n$. If $\C'=\text{sVec}$, then $\C$ is nilpotent by \cite[Corollary 3.12]{Yu} when $p$ is odd; if $p=2$, then previous argument shows that $\C$ contains  $\C(\Z_d,q)$ as a fusion subcategory, consequently $\C\cong \D\boxtimes \C(\Z_d,q)$, where M\"{u}ger center $\D'\subseteq \text{sVec}$ and $\FPdim(\D)=2^n$. Since  $\C$  is braided equivalent to a Deligne tensor product fusion subcategories of prime power Frobenius-Perron dimensions, it is nilpotent by \cite[Theorem 8.28]{ENO1}, and $\C$ is group-theoretical \cite[Corollary 6.8]{DrGNO1}.
\end{proof}

\author{Victor Ostrik \\ \thanks{Email:\,vostrik@uoregon.edu}
\\{\small Department  of Mathematical, University of Oregon, Eugene 97402, USA}}\\
{\small Laboratory of Algebraic Geometry,
National Research University Higher School of Economics, Moscow, Russia}\\
\author{Zhiqiang Yu\\ \thanks{Email:\,zhiqyumath@yzu.edu.cn}
\\{\small School of Mathematical Sciences,  Yangzhou University, Yangzhou 225002, China}}
\end{document}